\documentclass[a4paper,12pt]{article}
\usepackage{amsmath}
\usepackage{amsthm}
\usepackage{amssymb}
\usepackage{framed}
\usepackage[inline]{enumitem}
\usepackage{color}
\usepackage{xfrac}
\usepackage[all]{xy}
\usepackage[dvipsnames]{xcolor}
\usepackage[utf8]{inputenc}
\usepackage{setspace}
\usepackage{currfile}
\usepackage{comment}
\usepackage{mathabx}
\usepackage{textcomp}
\usepackage{cancel}
\usepackage{mathrsfs}
\usepackage[unicode]{hyperref}
\usepackage{cleveref}
\usepackage{thmtools}
\usepackage{booktabs}
\usepackage[ % BIBER
sorting=nyt,
style=numeric-comp,
maxbibnames=4,
firstinits=true,
block=space,
backend=biber
]{biblatex}
\usepackage{newunicodechar}
\newunicodechar{⟨}{\langle}
\newunicodechar{⟩}{\rangle}

\usepackage[unicode]{hyperref}

\renewbibmacro{in:}{%
  \ifentrytype{article}{}{\printtext{\bibstring{in}\intitlepunct}}}

\DeclareSourcemap{
  \maps[datatype=bibtex, overwrite]{
    \map{
      \step[fieldset=language, null]
      \step[fieldset=shorthand, null]
    }
  }
}

\DeclareFieldFormat[article]{title}{#1}
\DeclareFieldFormat[incollection]{title}{#1}

\DeclareFieldFormat[unpublished]{note}{#1\addcomma\space}
\DeclareFieldFormat[thesis]{type}{#1\addcomma\space}

\DeclareFieldFormat{pages}{#1}
\DeclareFieldFormat[book]{pages}{}

\DeclareBibliographyDriver{unpublished}{%
  \usebibmacro{bibindex}%
  \usebibmacro{begentry}%
  \usebibmacro{author}%
  \setunit{\labelnamepunct}\newblock
  \usebibmacro{title}%
  \newunit
  \printlist{language}%
  \newunit\newblock
  \usebibmacro{byauthor}%
  \newunit\newblock
  \printfield{howpublished}%
  \newunit\newblock
  \printfield{note}%
  \newunit\newblock
  \usebibmacro{location+date}%
  \newunit\newblock
  \iftoggle{bbx:url}
  {\usebibmacro{doi+eprint+url}}
  {}%
  \newunit\newblock
  \usebibmacro{addendum+pubstate}%
  \setunit{\bibpagerefpunct}\newblock
  \usebibmacro{pageref}%
  \newunit\newblock
  \iftoggle{bbx:related}
  {\usebibmacro{related:init}%
    \usebibmacro{related}}
  {}%
  \usebibmacro{finentry}}

\hypersetup{
  colorlinks,
  linkcolor={blue!80!black},
  citecolor={blue!80!black},
  urlcolor={blue!80!black}
}

\DeclareFontFamily{U}{mathx}{\hyphenchar\font45}
\DeclareFontShape{U}{mathx}{m}{n}{
  <5> <6> <7> <8> <9> <10>
  <10.95> <12> <14.4> <17.28> <20.74> <24.88>
  mathx10
}{}
\DeclareSymbolFont{mathx}{U}{mathx}{m}{n}
\DeclareFontSubstitution{U}{mathx}{m}{n}
\DeclareMathAccent{\widecheck}{0}{mathx}{"71}
\DeclareMathAccent{\wideparen}{0}{mathx}{"75}
\DeclareMathOperator{\Pf}{Pf}

\newcommand{\coleq}{\mathrel{\mathop:}=}

\newcommand{\bine}{\mathbin{\varepsilon}}
\newcommand{\e}{\varepsilon}

\declaretheorem[name=Theorem,refname={theorem,theorems},Refname={Theorem,Theorems}]{theorem}

\declaretheorem[name=Proposition,refname={proposition,propositions},Refname={Proposition,Propositions},sibling=theorem]{proposition}
\declaretheorem[name=Corollary,refname={corollary,corollaries},Refname={Corollary,Corollaries},sibling=theorem]{corollary}
\declaretheorem[name=Lemma,refname={lemma,lemmas},Refname={Lemma,Lemmas},sibling=theorem]{lemma}
\declaretheorem[name=Remark,refname={remark,remarks},Refname={Remark,Remarks},sibling=theorem]{remark}
\declaretheorem[name=Remarks,refname={remark,remarks},Refname={Remark,Remarks},sibling=theorem]{remarks}

\crefname{proposition}{proposition}{propositions}

\bibliography{dtk}

\newcommand{\cB}{{\mathscr{B}}}
\newcommand{\bZ}{{\mathbb{Z}}}
\newcommand{\cC}{\mathscr{C}}
\newcommand{\cD}{\mathscr{D}}
\newcommand{\cE}{\mathscr{E}}
\newcommand{\cF}{\mathscr{F}}
\newcommand{\cH}{\mathscr{H}}
\newcommand{\cU}{\mathscr{U}}
\newcommand{\cJ}{\mathscr{J}}
\newcommand{\cL}{\mathscr{L}}
\newcommand{\cM}{\mathscr{M}}
\newcommand{\cO}{\mathscr{O}}
\newcommand{\cS}{\mathscr{S}}
\newcommand{\bN}{\mathbb{N}}
\newcommand{\bR}{\mathbb{R}}
\newcommand{\bC}{\mathbb{C}}
\newcommand{\pd}{\partial}
\newcommand{\ud}{\mathrm{d}}

\newcommand{\abso}[1]{\left|#1\right|}
\newcommand{\norm}[1]{\left\lVert#1\right\rVert}

\title{A simpler description of the $\kappa$-topologies on the spaces $\cD_{L^p}$, $L^p$, $\cM^1$}
\author{Christian Bargetz\footnote{Institut für Mathematik, Universität Innsbruck, Technikerstraße 13, 6020 Innsbruck, Austria. e-mail: \href{mailto:christian.bargetz@uibk.ac.at}{christian.bargetz@uibk.ac.at}.}, Eduard A.~Nigsch\footnote{Wolfgang Pauli Institut, Oskar-Morgenstern-Platz 1, 1090 Wien, Austria. e-mail: \href{mailto:eduard.nigsch@univie.ac.at}{eduard.nigsch@univie.ac.at}.}, Norbert Ortner\footnote{Institut für Mathematik, Universität Innsbruck, Technikerstraße 13, 6020 Innsbruck, Austria. e-mail: \href{mailto:mathematik1@uibk.ac.at}{mathematik1@uibk.ac.at}.}}
\date{November 2017}

\begin{document}
\maketitle

\begin{abstract}
The $\kappa$-topologies on the spaces $\cD_{L^p}$, $L^p$ and $\cM^1$ are defined by a neighbourhood basis consisting of polars of absolutely convex and compact subsets of their (pre-)dual spaces. In many cases it is more convenient to work with a description of the topology by means of a family of semi-norms defined by multiplication and/or convolution with functions and by classical norms. We give such families of semi-norms generating the $\kappa$-topologies on the above spaces of functions and measures defined by integrability properties. In addition, we present a sequence-space representation of the spaces $\cD_{L^p}$ equipped with the $\kappa$-topology, which complements a result of J.~Bonet and M.~Maestre. As a byproduct, we give a characterisation of the compact subsets of the spaces $\cD'_{L^p}$, $L^p$ and $\cM^1$.
\end{abstract}

{\bfseries Keywords: }{locally convex distribution spaces, topology of uniform convergence on compact sets, $p$-integrable smooth functions, compact sets}

{\bfseries MSC2010 Classification: }{46F05, 46E10, 46A13, 46E35, 46A50, 46B50}

\section{Introduction}\label{sec1}

In the context of the convolution of distributions, L.~Schwartz introduced the spaces $\cD_{L^p}$ of $C^\infty$-functions whose derivatives are contained in $L^p$ and the spaces $\cD'_{L^q}$ of finite sums of derivatives of $L^q$-functions, $1 \le p,q \le \infty$. The topology of $\cD_{L^p}$ is defined by the sequence of \mbox{(semi-)}norms
\[
  \cD_{L^p} \to \bR_+,\quad \varphi \mapsto p_m(\varphi) \coleq \sup_{\abso{\alpha} \le m} \norm{\pd^\alpha \varphi}_p,
\]
whereas $\cD'_{L^q} = (\cD_{L^p})'$ for $\sfrac{1}{p} + \sfrac{1}{q} = 1$ if $p< \infty$ and $\cD'_{L^1} = (\dot\cB)'$, carry the strong dual topology.

Equivalently, for $1 \le p < \infty$, by barrelledness of these spaces, the topology of $\cD_{L^p}$ is also the topology $\beta(\cD_{L^p}, \cD'_{L^q})$ of uniform convergence on the \emph{bounded sets} of $\cD'_{L^q}$. In the case of $p=1$, the duality relation $(\dot{\cB}')'=\cD_{L^1}$, see~\cite[Proposition~7, p.~13]{Ba}, provides that $\cD_{L^1}$ also has the topology of uniform convergence on bounded sets of $\dot{\cB}'$. Following L.~Schwartz in~\cite{Sch1}, $\dot{\cB}'$ denotes the closure of $\cE'$ in $\cB' = \cD_{L^\infty}'$ which is not the dual space of $\dot{\cB}$ but in some sense is its analogon for distributions instead of smooth functions. For $p = \infty$, $\cD_{L^\infty}$ carries the topology $\beta(\cD_{L^\infty}, \cD'_{L^1})$ due to $(\cD'_{L^1})' = \cD_{L^\infty}$ \cite[p.~200]{Sch1}.

By definition, the topology of $\cD'_{L^q}$ is the topology of uniform convergence on the bounded sets of $\cD_{L^p}$ and $\dot\cB$ for $q>1$ and $q=1$, respectively.

In \cite{Sch4}, the spaces $\cD_{L^p, c}$ and $\cD'_{L^q, c}$ are considered{, where the index $c$ designates the topologies $\kappa(\cD_{L^p}, \cD'_{L^q})$ and $\kappa(\cD'_{L^q}, \cD_{L^p})$ of uniform convergence on \emph{compact} sets of $\cD'_{L^q}$ and $\cD_{L^p}$, respectively. Let us mention three reasons why the spaces $L^q_c$ and $\cD'_{L^q,c}$ are of interest:
\begin{enumerate}[label=(\arabic*)]
\item Let $E$ and $F$ be distribution spaces. If $E$ has the $\e$-property, a kernel distribution $K(x,y) \in \cD'_x(F_y)$ belongs to the space $E_x(F_y)$ if it does so ``scalarly'', i.e.,
  \[
    K(x,y) \in E_x(F_y) \Longleftrightarrow \forall f \in F': \langle K(x,y), f(y) \rangle \in E_x.
  \]
  The spaces $L^q_c$, $1 < q \le \infty$, and $\cD'_{L^q,c}$, $1 \le q \le \infty$, have the $\e$-property \cite[Proposition 16, p.~59]{Sch4} whereas for $1 \le p \le \infty$, $L^p$ and $\cD'_{L^p}$ do not. For $L^p$, $1\leq p<\infty$, this can be seen by checking that $K(x,k)=(k^{1/p}/(1+k^2x^2)) \in \cD'(\mathbb{R})\mathbin{\widehat{\otimes}} c_0$ satisfies $\langle K(x,k),a\rangle \in L^p$ for all $a\in \ell^1$ but is not contained in $L^p\mathbin{\widehat{\otimes}}_\varepsilon c_0$; in the case of $p = \infty$ one takes $K(x,k) = Y(x-k)$ instead. For $\cD_{L^q}'$ a similar argument with $K(x,k)=\delta(x-k)$ can be used.
\item The kernel $\delta(x-y)$ of the identity mapping $E_x \to E_y$ of a distribution space $E$ belongs to $E'_{c,y} \bine E_x$. Thus, e.g., the equation 
  \[
    \delta(x-y) = \sum_{k=0}^\infty H_k(x) H_k(y)
  \]
  where the $H_k$ denote the Hermite functions, due to G.~Arfken in~\cite[p.~529]{Arfken} and P.~Hirvonen in~\cite{Hirvonen}, is valid in the spaces $L^q_{c,y} \mathbin{\widehat\otimes_\e} L^p_x$, $\cD'_{L^q,c,y} \mathbin{\widehat \otimes_\e} \cD_{L^p,x}$ if $1 \le p < \infty$, or $\cM^1_{c,y} \mathbin{\widehat\otimes_\e} \cC_{0,x}$, see~\cite{NOAH}, or, classically, in $\cS'_y \mathbin{\widehat\otimes} \cS_x$.
\item The classical Fourier transform
  \[
    \cF \colon L^p \to L^q,\qquad 1 \le p \le 2,\ \sfrac{1}{p} + \sfrac{1}{q} = 1,
  \]
  is well-defined and continuous by the Hausdorff-Young theorem. By means of the kernel $e^{-ixy}$ of $\cF$ we can express the Hausdorff-Young theorem by
  \begin{align*}
    e^{-ixy} \in \cL_c ( L^p_x, L^q_y ) & \cong L_\e ( L^p_{c,y} , L^q_{c,x} ) = L^q_y \mathbin{\widehat{\otimes}}_\varepsilon L^q_{c,x} \\
                                        & = L^q_{c,x} \mathbin{\widehat{\otimes}}_\varepsilon L^q_y
  \end{align*}
  if $p>1$, and by $e^{-ixy} \in L^\infty_{c,x} ( \cC_{0,y} )$ if $p=1$. Obvious generalizations are
  \begin{align*}
    e^{-ixy} & \in \cD_{L^q, c, x} ( \bigcup_{k=0}^\infty (L^q)_{k,y} ), \quad 2<q<\infty,\textrm{ and }\\
    e^{-ixy} & \in \cD_{L^\infty,c,x}( \cO^0_{c,y} ).
  \end{align*}
\end{enumerate}

Whereas the topology of $\cD_{L^p}$, $1\leq p\leq \infty$, can be described either by the seminorms $p_m$ or, equivalently, by
\[
  p_B \colon \cD_{L^p} \to \bR,\quad p_B(\varphi) = \sup_{S \in B} \abso{ \langle \varphi, S \rangle },\quad B \subseteq \cD'_{L^q}\textrm{ bounded},
\]
the topology of $\cD_{L^p, c}$ only is described by
\[
  p_C \colon \cD_{L^p} \to \bR,\quad p_C(\varphi) = \sup_{S \in C} \abso{\langle \varphi, S \rangle },\quad C \subseteq \cD'_{L^q}\textrm{ compact}
\]
for $1 < p \le \infty$ and
\[
  p_C \colon \cD_{L^1} \to \bR,\quad p_C(\varphi) = \sup_{S \in C} \abso{\langle \varphi, S \rangle },\quad C \subseteq \dot{\cB}' \textrm{ compact}.
\]
An analogue statement holds for $L^p_c$, $1 < p \le \infty$ and $\cM^1_c$.

Thus, our task is the description of the topologies $\kappa(\cD_{L^p}, \cD'_{L^q})$, $\kappa(\cD_{L^1}, \dot{\cB}')$, $\kappa(L^p, L^q)$ and $\kappa(\cM^1, \cC_0)$ (for $1<p\leq \infty$) by seminorms involving functions and not sets (\Cref{prop3p1,prop3,prop8,prop9}). As a byproduct, the compact sets in $\cD'_{L^q}$ and $L^q$ for $1 \le q < \infty$ are characterised in \Cref{prop2} and in \Cref{cor8}, respectively. In addition, we also give characterisations of the compact sets of $\dot{\cB}'$ and $\cM^1$.

The notation generally adopted is the one from~\cite{Sch5, Sch4, Sch1, Sch2, Sch3} and \cite{H}. However, we deviate from these references by defining the \emph{Fourier transform} as
\[
  \cF \varphi(y) = \int_{\bR^n} e^{-iyx} \varphi(x)\,\ud x,\qquad \varphi \in \cS(\bR^n) = \cS.
\]

We follow~\cite[p.~36]{Sch1} in denoting by $Y$ the Heaviside-function. The translate of a function $f$ by a vector $h$ is denoted by $(\tau_h f)(x) \coleq f(x-h)$.} Besides the spaces $L^p(\bR^n) = L^p$, $\cD_{L^p}(\bR^n) = \cD_{L^p}$, $1 \le p \le \infty$, we use the space $\cM^1(\bR^n) = \cM^1$ of integrable measures which is the strong dual of the space $\cC_0$ of continuous functions vanishing at infinity. In measure theory the measures in~$\cM^{1}$ usually are called bounded measures whereas J.~Horv\'{a}th, in analogy to the integrable distributions, calls them integrable measures. Here we follow J.~Horv\'{a}th's naming convention.

The \emph{topologies} on Hausdorff locally convex spaces $E,F$ we use are
\begin{description}
\item[$\beta(E,F)$] -- the topology (on $E$) of uniform convergence on \emph{bounded} sets of $F$,
\item[$\kappa(E,F)$] -- the topology (on $E$) of uniform convergence on \emph{absolutely convex compact} subsets of $F$, see~\cite[p.~235]{H}.
\end{description}
Thus, if $1 < p \le \infty$ and $\sfrac{1}{p} + \sfrac{1}{q} = 1$,
\begin{align*}
  \cM^1_c &= ( \cM^1, \kappa (\cM^1, \cC_0 )), \\
  L^p_c &= (L^p, \kappa ( L^p, L^q)), \\
  \cD_{L^1,c} &= (\cD_{L^1}, \kappa ( \cD_{L^1}, \dot{\cB}' ),\\
  \cD_{L^p, c} &= ( \cD_{L^p}, \kappa ( \cD_{L^p}, \cD'_{L^q}) ). %, \;\text{for}\;1 < p <\infty, \\
  %\cB_c &= \cD_{L^\infty, c} = ( \cD_{L^\infty}, \kappa( \cD_{L^\infty}, \cD'_{L^1})), \\
\end{align*}
We use the spaces
\[
  H^{s,p} = \cF^{-1} ( (1+\abso{x}^2)^{-s/2} \cF L^p ) \qquad 1 \le p \le \infty, s\in\mathbb{R},
\]
see~\cite[Def~3.6.1, p.~108]{DVAF}, \cite[7.63, p.~252]{Adams} or~\cite{Wo}. The weighted $L^p$-spaces are
\[
  L^p_k = (1 + \abso{x}^2)^{-k/2} L^p,\quad k \in \bZ.
\]

In addition, we consider the following sequence spaces. By $s$ we denote the space of rapidly decreasing sequences, by $s'$ its dual, the space of slowly increasing sequences. Moreover we consider the space $c_0$ of null sequences and its weighted variant
\[
  (c_0)_{-k} = \{ x \in \bC^{\bN} \colon \lim_{j \to \infty} j^{-k} x(j) = 0 \}.
\]

The space $s'$ is the non-strict inductive limit
\[
  s' = \varinjlim_{k} (c_0)_{-k}
\]
with compact embeddings, i.e.~an (LS)-space, see~\cite[p.~132]{FW}.

The uniquely determined temperate \emph{fundamental solution} of the iterated metaharmonic operator $(1-\Delta_n)^k$, where $\Delta_n$ is the $n$-dimensional Laplacean, is given by
\[
  L_{2k} = \frac{1}{2^{k+n/2-1} \pi^{n/2} \Gamma(k)} \abso{x}^{k-n/2} K_{n/2 - k} ( \abso{x} ) = \cF^{-1} ( (1+\abso{x}^2)^{-k} )
\]
wherein the distribution
\[
  L_s = \frac{1}{2^{\frac{s+n}{2}-1} \pi^{n/2} \Gamma(s/2)} \abso{x}^{\frac{s-n}{2}} K_{\frac{n-s}{2}} ( \abso{x} )
\]
is defined by analytic continuation with respect to $s \in \bC$. The symbol $\Pf$ in \cite[(II, 3; 20), p.~47]{Sch1} is not necessary because $L_s \in \cH(\bC_s) \mathbin{\widehat\otimes} \cS'$ is an entire holomorphic function with values in $\cS'(\bR^n_x)$ \cite[p.~47]{Sch1}. In virtue of
\[
  \cF L_s = (1 + \abso{x}^2)^{-s/2} \in \cO_{M,x}
\]
we, in fact, have $L_s \in \cH(\bC_s) \mathbin{\widehat\otimes} \cO_C'$ and $L_s * L_t = L_{s+t}$ for $s,t \in \bC$ \cite[(VI, 8; 5), p.~204]{Sch1}.

Particular cases of this formula are:
\[
  L_0 = \delta,\quad L_{-2k} = (1 - \Delta_n)^k \delta, \quad (1-\Delta_n)^k L_{2k} = \delta,\quad k \in \bN.
\]
For $s>0$, $L_s>0$ and $L_s$ decreases exponentially at infinity. Moreover, $L_s \in L^1$ for $\mathop{\mathrm{Re}} s >0$. In contrast to \cite[p.~131]{St} and \cite{Wo}, we maintain the original notation $L_s$ for the Bessel kernels and we write $L_s *$ instead of $\cJ_s$. The spaces $\cD'_{L^q}$ can be described as the inductive limit
\[
  \bigcup_{m=0}^\infty (1-\Delta_n)^m L^q = \varinjlim_m H^{-2m,q},
\]
see, e.g., \cite[p.~205]{Sch1}. The space $\dot{\cB}'$ of distributions vanishing at infinity has the similar representation
\[
  \dot{\cB}' = \bigcup_{m=0}^\infty (1-\Delta_n)^m \cC_0 = \varinjlim_m (1-\Delta_n)^m \cC_0,
\]
by~\cite[p.~205]{Sch1}. If we equip $(1-\Delta_n)^m \cC_0$ with the final topology with respect to $(1-\Delta_n)^m$ an application of de Wilde's closed graph theorem provides the topological equality since by~\cite[Proposition~7, p.~65]{Ba} $\dot{\cB}'$ is ultrabornological and $\varinjlim_m (1-\Delta_n)^m \cC_0$ has a completing web since it is a Hausdorff inductive limit of Banach spaces.

\section{``Function''-seminorms in $\cD_{L^p,c}$\,, $1 < p \le \infty$}\label{sec2}

In order to describe the topology of $\cD_{L^p,c}$ $(1 < p < \infty$) by ``function''-seminorms it is necessary to characterise the compact sets of the dual space $\cD'_{L^q}$ ($1 < q < \infty$), defined in \cite[p.~200]{Sch1} as
\[
  \cD'_{L^q} = (\cD_{L^p})',\quad \frac{1}{p} + \frac{1}{q} = 1,
\]
and endowed with the strong topology $\beta(\cD'_{L^q}, \cD_{L^p})$. The description of $\cD_{L^\infty,c}$ is already well-known \cite{DD}.

Due to the definition of the space $\cD_{L^p}$, $1 \le p < \infty$, as the countable projective limit $\bigcap_{m=0}^\infty H^{2m,p}$ of the Banach spaces $H^{2m,p}$, which are called ``potential spaces'' in \cite[p.~135]{St}, we conclude that the strong dual $\cD'_{L^q}$ coincides with the countable inductive limit $\bigcup_{m=0}^\infty H^{-2m,q}$. Note that the topological identity follows from the ultrabornologicity of $(\cD_{L^q}', \beta(\cD'_{L^q}, \cD_{L^p}))$, which follows for example by the sequence-space representation $\cD_{L^p}'\cong s'\mathbin{\widehat{\otimes}}\ell^q$ given independently by D.~Vogt in~\cite{Vogt2} and by M.~Valdivia in~\cite{Va}, by means of Grothendiecks Th\'eor\`eme B \cite[p.~17]{G}. The completeness of $\cD'_{L^q}$ implies the regularity of the inductive limit $\bigcup_{m=0}^\infty H^{-2m,q}$ \cite[p.~77]{B}.

An alternative proof of the representation of $\cD'_{L^q}$ as the inductive limit of the potential spaces above can be given using~\cite[Theorem~5]{BW} and the fact that $1-\Delta_n$ is a densely defined and invertible, closed operator on $L^q$.

We first show that the (LB)-space $\cD'_{L^q}$, $1 \le q < \infty$ is compactly regular \cite[6.~Definition (c), p.~100]{B}:

\begin{proposition}\label{prop1}
  If $1 \le q < \infty$ the (LB)-space $\cD'_{L^q}$ is compactly regular.
\end{proposition}
\begin{proof}
  Compactly regular (LF)-spaces are characterised by condition (Q) (\cite[Thm.~2.7, p.~252]{W2}, \cite[Thm.~6.4, p.~112]{W1}) which in our case reads as:
  \begin{gather*}
    \forall m \in \bN_0\ \exists k >m\ \forall \e>0\ \forall \ell>k\ \exists C>0:\\
    \norm{S}_{2k, q} \le \e \norm{S}_{2m, q} + C \norm{S}_{2\ell, q}\qquad \forall S \in H^{-2m, q}.
  \end{gather*}
  (Note that $H^{-2m, q} \hookrightarrow H^{-2k, q} \hookrightarrow H^{-2\ell, q}$.)
  For (Q) see \cite[Prop.~2.3, p.~62]{V}. By definition,
  \[
    \norm{L_{2k} * S}_q \le \e \norm{L_{2m} * S}_q + C \norm{L_{2\ell} * S}_q\quad \forall S \in H^{-2m, q}
  \]
  is equivalent to
  \[
    \norm{L_{2(k-\ell)} * S}_q \le \e \norm{L_{2(m-\ell)} * S}_q + C \norm{S}_q\quad \forall S \in H^{-2(m-\ell), q}
  \]
  But this inequality follows from Ehrling's inequality \cite{Wo}, which states that for $1 \le q < \infty$ and $0 < s < t$,
  \begin{align*}
    \forall \e>0\ \exists C>0:\ \norm{\cJ_s \varphi}_q \le \e \norm{\varphi}_q + C \norm{\cJ_t \varphi}_q, \varphi \in \cS.
  \end{align*}
  By density of $\cS$ in $H^{-2(m-\ell),q}$ this implies the validity of (Q).
\end{proof}

\begin{remarks}\label{rem1}
  \begin{enumerate}[label=(\alph*)]
  \item\label{rem1a} By means of M.~Valdivia's and D.~Vogt's sequence space representation $\cD_{L^p} \cong s \mathbin{\widehat\otimes} \ell^p$ given in~\cite[Thm.~1, p.~766]{Va}, and~\cite[(3.2) Theorem, p.~415]{Vogt2} and $\dot{\cB}\cong c_0\mathbin{\widehat{\otimes}}s$, we obtain by \cite[Chapter II, Thm.~XII, p.~76]{G} that $\cD'_{L^q}\cong \ell^q\mathbin{\widehat{\otimes}}s'$, $1 \le q \le \infty$. Using this representation, a further proof of the compact regularity of the (LB)-space $\cD'_{L^q}$ is given in Section \ref{sec4}.
  \item Differently, the compact regularity of the (LB)-space $\cD'_{L^1}$ is proven in \cite[(3.6) Prop., p.~71]{DD}.
  \item If $q=2$, the space $\cD'_{L^2}$ is isomorphic to the (LB)-space $\bigcup_{k=0}^\infty (L^2)_{-k}$. The compact regularity of the (LB)-spaces
    \[ \bigcup_{k=0}^\infty (L^p)_{-k},\qquad 1 \le p \le \infty, \]
    immediately follows from the validity of condition (Q). For $p=1$ the compact regularity of the space $\bigcup_{k=0}^\infty (L^1)_{-k}$ was shown in \cite[(3.8), Satz (a), (b), p.~28; (3.9) Bem., (a), p.~29]{D}.
  \end{enumerate}
\end{remarks}

The next proposition characterises compact sets in $\cD'_{L^q}$.

\begin{proposition}\label{prop2} Let $1 \le q < \infty$. A set $C \subseteq \cD'_{L^q}$ is compact if and only if for some $m \in \bN_0$, $L_{2m} * C$ is compact in $L^q$.
\end{proposition}

\begin{proof}
  ``$\Leftarrow$'': The compactness of $L_{2m} * C$ in $L^q$ implies its compactness in $\cD'_{L^q}$ and, hence, $C = L_{-2m} * (L_{2m} * C)$ is compact in $\cD'_{L^q}$.

  ``$\Rightarrow$'': In virtue of Proposition \ref{prop1} there is $m \in \bN_0$ such that $C$ is compact in $H^{-2m, q}$. The continuity of the mapping $\varphi \mapsto L_{2m} *\varphi$, $H^{-2m, q} \to L^q$ implies the compactness of $L_{2m} * C$ in $L^q$.
\end{proof}

The following proposition generalizes the description of the topology of the space $\cB_c = (\cD_{L^\infty}, \kappa ( \cD_{L^\infty}, \cD'_{L^1}))$ by the ``function''-seminorms
\[ p_{g,m}(\varphi) = \sup_{\abso{\alpha} \le m} \norm{g \pd^\alpha \varphi}_\infty,\quad g \in \cC_{0},\ m \in \bN_0 \]
for $\varphi \in \cB = \cD_{L^\infty}$ in \cite[(3.5) Cor., p.~71]{DD}. 

\begin{proposition}\label{prop3} Let $1 < p \le \infty$ and $\sfrac{1}{p} + \sfrac{1}{q} = 1$. The topology $\kappa(\cD_{L^p}, \cD'_{L^q})$ of $\cD_{L^p, c}$ is generated by the seminorms
  \[
    \cD_{L^p} \to \bR_+,\quad \varphi \mapsto p_{g,m}(\varphi) \coleq \norm{ g ( 1 - \Delta_n)^m \varphi}_p,\quad g \in \cC_{0}, m \in \bN_0,
  \]
  or equivalently by
  \[
    \varphi \mapsto \sup_{\abso{\alpha} \le m} \norm{g \pd^\alpha \varphi}_p,\quad g \in \cC_{0}, m \in \bN_0.
  \]
\end{proposition}

\begin{proof}Due to \cite[(3.5) Cor., p.~71]{DD} it suffices to assume $1 < p < \infty$. We denote the topology on $\cD_{L^p}$ generated by $\{ p_{g,m}:g \in \cC_{0},m\in\bN_0\}$ by $t$. Moreover, $B_{1,p}$ shall denote the unit ball in $L^p$.
  
  (a) $t \subseteq \kappa ( \cD_{L^p}, \cD'_{L^q} )$:

  If $\cU_{g,m} \coleq \{ \varphi \in \cD_{L^p}: p_{g,m}(\varphi) \le 1 \}$ is a neighborhood of $0$ in $t$ we have $\cU_{g,m} = ( ( \cU_{g,m})^\circ )^\circ$ by the theorem on bipolars. We show that $\cU^\circ_{g,m}$ is a compact set in $\cD'_{L^q}$. We have
  \begin{align*}
    \varphi \in \cU_{g,m} &\Longleftrightarrow g ( L_{-2m} * \varphi) \in B_{1,p} \\
                          & \Longleftrightarrow \sup_{\psi \in B_{1,q}} \abso{ \langle \psi, g ( L_{-2m} * \varphi ) \rangle } \le 1 \\
                          & \Longleftrightarrow \sup_{\psi \in B_{1,q}} \abso{ \langle L_{-2m} * (g \psi), \varphi \rangle } \le 1 \\
                          & \Longleftrightarrow \varphi \in (L_{-2m} * (g B_{1,q}))^\circ.
  \end{align*}
  Hence, $\cU^\circ_{g,m} = \overline{L_{-2m} * (g B_{1,q} )} \subseteq \cD'_{L^q}$. By Proposition \ref{prop2}, $\cU^\circ_{g,m}$ is compact in $\cD'_{L^q}$ if there exists $\ell \in \bN_0$ such that
  \[ L_{2 \ell} * \cU_{g,m}^\circ = \overline{L_{2(\ell-m)} * (g B_{1,q})} \]
  is compact in $L^q$. Choosing any $\ell>m$, it suffices to show that $C \coleq L_{2(\ell - m)} * (g B_{1,q})$ satisfies the three criteria of the M.~Fr\'echet-M.~Riesz-A.~Kolmogorov-H.~Weyl Theorem \cite[Thm.~6.4.12, p.~140]{Sch3}:
  \begin{enumerate}[label=(\roman*)]
  \item Because $\ell>m$, $\mu \coleq L_{2(\ell-m)} \in L^1$ and hence, for $\varphi \in B_{1,q}$,
    \[ \norm{ \mu * (g \varphi)}_q \le \norm{ \mu }_1 \norm{g}_{\infty}, \]
    i.e., $C$ is bounded in $L^q$.
  \item The set $C$ has to be small at infinity: for $\varphi \in B_{1,q}$, 
    \begin{align*}
      \|Y ( \abso{.} - R)&(\mu * (g \varphi ) )\|_q \le\\
                         &\le \norm{ Y ( \abso{.} - R) \left( ( Y ( \frac{R}{2} - \abso{.})\mu)*(g \varphi) \right)}_q\\
                         & \qquad + \norm{ Y ( \abso{.} - R) \left( ( Y ( \abso{.} - \frac{R}{2})\mu)*(g \varphi) \right)}_q \\
                         &\le \left( \int_{\abso{x} \ge R} \left| \int_{\abso{\xi} \le R/2} \mu(\xi) (g \varphi)(x-\xi)\,\ud \xi\right|^q\,\ud x \right)^{1/q}\\
                         & \qquad + \norm{ ( Y ( \abso{.} - R/2)\mu)*(g \varphi)}_q \\
                         &\le \int_{\abso{\xi} \le R/2} \left( \int_{\abso{x} \ge R} \abso{ \mu(\xi) (g \varphi)(x-\xi)}^q\,\ud x\right)^{1/q}\,\ud \xi\\
                         & \qquad + \norm{ Y ( \abso{.} - R/2 )\mu}_1 \cdot \norm{g}_\infty \\
                         &\le \int_{\abso{\xi} \le R/2} \abso{ \mu(\xi) } \,\ud \xi \cdot \left( \int_{\abso{z} \ge R/2} \abso{ (g \varphi)(z)}^q \,\ud z\right)^{1/q}\\
                         & \qquad + \norm{ Y ( \abso{.} - R/2)\mu}_1 \cdot \norm{g}_\infty \\
                         &\le \norm{\mu}_1 \norm{ Y ( \abso{.} - R/2) g }_\infty\\
                         & \qquad + \norm{Y(\abso{.} - R/2) \mu}_1 \cdot \norm{g}_\infty.
    \end{align*}
    Hence, $\lim_{R \to \infty} \norm{ Y ( \abso{.} - R)( \mu * (g \varphi ))}_q = 0$ uniformly for $\varphi \in B_{1,q}$.
  \item $C$ is $L^q$-equicontinuous because
    \begin{gather*}
      \norm{\tau_h ( \mu * (g\varphi)) - \mu * (g\varphi) }_q \le \norm{\tau_h \mu - \mu }_1 \norm{g}_\infty
    \end{gather*}
    tends to $0$ if $h \to 0$, uniformly for $\varphi \in B_{1,q}$.
  \end{enumerate}
  (b) $\kappa(\cD_{L^p}, \cD'_{L^q}) \subset t$:

  If $C^\circ$ is a $\kappa(\cD_{L^p}, \cD'_{L^q})$-neighborhood of $0$ with $C$ a compact set in $\cD'_{L^q}$ then, by \Cref{prop2}, there exists $m \in \bN_0$ such that the set $L_{2m} * C$ is compact in $L^q$. By means of Lemma \ref{lem1} below there is a function $g \in \cC_{0}$ such that $L_{2m} * C \subseteq g B_{1,q}$, hence $C \subseteq L_{-2m} * (g B_{1,q}) \subseteq \cU^\circ_{g,m}$. Thus, $C^\circ \supseteq \cU_{g,m}$, i.e., $C^\circ$ is a neighborhood in $t$.
\end{proof}

\begin{lemma}\label{lem1}Let $1 \le q < \infty$. 
  If $K \subseteq L^q$ is compact then there exists $g \in \cC_{0}$ such that $K \subseteq g B_{1,q}$.
\end{lemma}

\begin{proof}
  Apply the Cohen-Hewitt factorization theorem \cite[(17.1), p.~114]{DW} to the bounded subset $K$ of the (left) Banach module $L^q$ with respect to the Banach algebra $\cC_0$ having (left) approximate identity $\{ e^{-k^2 \abso{x}^2 } : k >0 \}$.
\end{proof}

\begin{remark}
  Denoting by $\tau(\cB, \cD'_{L^1})$ the Mackey-topology on $\cB = \cD_{L^\infty}$ we even have for $p = \infty$, $\cB = \cD_{L^\infty}$:
  \[
    \cB_c = (\cB, \kappa ( \cB, \cD'_{L^1} ) ) = (\cB, t) = ( \cB, \tau ( \cB, \cD'_{L^1} ) ),
  \]
  since $\cD'_{L^1}$ is a Schur space \cite[p.~52]{DD}.
\end{remark}

\section{The case $p=1$}

Using the sequence-space representation $\dot{\cB}'\cong s'\mathbin{\widehat{\otimes}}c_0$ given in~\cite[Theorem~3, p.~13]{Ba}, the compact regularity of the (LB)-space $\dot{\cB}'$ can be shown similarly to the proof of \Cref{prop6}. Moreover, one has the following characterisation of the compact sets of $\dot{\cB}'$.

\begin{proposition}\label{prop2p1} A set $C \subseteq \dot{\cB}'$ is compact if and only if for some $m \in \bN_0$, $L_{2m} * C$ is compact in $\cC_0$.
\end{proposition}

\begin{proof}
  The proof is completely analogous to the one of \Cref{prop2}.
\end{proof}

\begin{proposition}\label{prop3p1} The topology $\kappa(\cD_{L^1}, \dot{\cB}')$ of $\cD_{L^1, c}$ is generated by the seminorms
  \[
    \cD_{L^1} \to \bR_+,\quad \varphi \mapsto p_{g,m}(\varphi) \coleq \norm{ g ( 1 - \Delta_n)^m \varphi}_1,\quad g \in \cC_{0}, m \in \bN_0,
  \]
  or equivalently by
  \[
    \varphi \mapsto \sup_{\abso{\alpha} \le m} \norm{g \pd^\alpha \varphi}_1,\quad g \in \cC_{0}, m \in \bN_0.
  \]
\end{proposition}

\begin{proof}
  We first show that the topology $t$ generated by the above seminorms is finer than the topology of uniform convergence on compact subsets of $\dot{\cB}'$.
  Similarly to the proof of~\Cref{prop3}, we have to show that the set $L_{2(\ell-m)}*(g B_{1,\cC_0})$ is a relatively compact subset of $\dot{\cB}'$, where $B_{1, \cC_0}$ denotes the unit ball of $\cC_0$. We pick $\ell> m$ and, by the compact regularity of $\dot{\cB}'$ and the Arzela-Ascoli theorem, we have to show that $L_{2(\ell-m)}*(g B_{1,\cC_0})$ is bounded as a subset of $\cC_0$ and equicontinous as a set of functions on the Alexandroff compactification of $\mathbb{R}^{n}$. Since $\ell>m$, we have that $L_{2(\ell-m)}\in L^1$. For every $\varphi\in B_{1,\cC_0}$, Young's convolution inequality implies
  \[
    \|L_{2(\ell-m)}*(g\varphi)\|_\infty \leq \|L_{2(\ell-m)}\|_1\|g\varphi\|_\infty \leq \|L_{2(\ell-m)}\|_1\|g\|_\infty,
  \]
  i.e. $L_{2(\ell-m)}*(gB_{1,\cC_0})\subseteq \cC_{0}$ is a bounded set.

  Since a translation of a convolution product can be computed by applying the translation to one of the factors, we can again use Young's convolution inequality to obtain
  \begin{align*}
    \|\tau_h(L_{2(\ell-m)}*(g\varphi))-L_{2(\ell-m)}*(g\varphi)\|_\infty & = \|(\tau_{h}L_{2(\ell-m)}-L_{2(\ell-m)})*(g\varphi)\|_\infty \\
     & \leq \|(\tau_{h}L_{2(\ell-m)}-L_{2(\ell-m)})\|_1 \|g\|_\infty
  \end{align*}
  for all $\varphi$ in the unit ball of $\cC_0$. From this inequality we may conclude that $L_{2(\ell-m)}*(g B_{1,\cC_0})$ is equicontinuous at all points in $\mathbb{R}^{n}$. Therefore we are left to show that it is also equicontinous at infinity. In order to do this, first observe that $|g|\in\cC_{0}$ and $|L_{2(\ell-m)}|=L_{2(\ell-m)}$. Moreover, Lebesgue's theorem on dominated convergence implies that the convolution of a function in $\cC_0$ and an $L^{1}$-function is contained in~$\cC_0$. Finally, by the above reasoning the inequality
  \[
    |(L_{2(\ell-m)}*(g\varphi))(x)| \leq \int_{\mathbb{R}^{n}} |g(x-\xi)| |L_{2(\ell-m)}(\xi)|\,\mathrm{d}\xi = (L_{2(\ell-m)}*|g|)(x)
  \]
  shows that $L_{2(\ell-m)}*(g B_{1,\cC_0})$ is equicontinuous at infinity.

  The proof that $\kappa(\cD_{L^1},\dot{\cB}')$ is finer than $t$ is completely analogous to the corresponding part of \Cref{prop3} if we can show that every compact subset of $\cC_0$ is contained in~$g B_{1,\cC_0}$ for some $g\in\cC_0$. Let $C\subseteq\cC_0$ be a compact set. By the Arzela-Ascoli theorem, it is equicontinuous at infinity, i.e., for every $k\in\mathbb{N}$ there is an $R_k$ such that $|h(x)|\leq 1/k$ for $h\in C$ and all $|x|>R_k$. This condition implies the existence of the required function $g\in\cC_0$ with the above property.
\end{proof}

\section{Properties of the spaces $\cD_{L^p,c}$}\label{sec3}

In \cite[p.~127]{Sch4}, L.~Schwartz proves that the spaces $\cB_c = \cD_{L^\infty,c}$ and $\cB = \cD_{L^\infty}$ have the same bounded sets, and that on these sets the topology $\kappa( \cB, \cD'_{L^1})$ equals the topology induced by $\cE = C^\infty$. Moreover, $\kappa(\cB, \cD'_{L^1})$ is the finest locally convex topology with this property. By an identical reasoning we obtain:

\begin{proposition}\label{prop4} Let $1 \le p < \infty$.
  \begin{enumerate}[label=(\alph*)]
  \item The spaces $\cD_{L^p}$ and $\cD_{L^p, c}$ have the same bounded sets. These sets are relatively $\kappa(\cD_{L^p}, \cD'_{L^q})$-compact and relatively $\kappa(\cD_{L^1}, \dot{\cB}')$-compact for $1<p<\infty$ and $p=1$, respectively.
  \item The topology $\kappa ( \cD_{L^p}, \cD'_{L^q} )$ of $\cD_{L^p, c}$ is the finest locally convex topology on $\cD_{L^p}$ which induces on bounded sets the topologies of $\cE$ or $\cD'$ or $\cD'_{L^p}$.
  \end{enumerate}
\end{proposition}

In the following proposition we collect some further properties of the spaces $\cD_{L^p}$:

\begin{proposition}Let $1 \le p \le \infty$.
  The spaces $\cD_{L^p, c}$ are complete, quasinormable, semi-Montel and hence semireflexive. $\cD_{L^p,c}$ is not infrabarrelled and hence neither barrelled nor bornological. $\cD_{L^p, c}$ is a Schwartz space but not a nuclear space.
\end{proposition}

\begin{proof}
  (1) The completeness follows either from \cite[(1), p.~385]{K1}, or from \cite[Ex.~7(b), p.~243]{H}. It must be taken into account that $\cD'_{L^q}$ and $\dot\cB'$ are bornological.

  (2) $\cD_{L^p,c}$ is quasinormable since its dual $\cD'_{L^q}$ is boundedly retractive (see the argument in \cite[p.~73]{DD} and use \cite[Def.~4, p.~106]{G3}) which, by \cite[7.~Prop., p.~101]{B} is equivalent with its compact regularity (Proposition \ref{prop1}).

  (3) Since bounded and relative compact sets coincide in $\cD_{L^p, c}$ it is a semi-Montel space.

  (4) Infrabarrelledness and the Montel-property would imply that $\cD_{L^p, c}$ is Montel which in turn implies the coincidence of the topologies $\kappa ( \cD_{L^p}, \cD'_{L^q})$ and $\beta(\cD_{L^p}, \cD'_{L^q})$. This is a contradiction, since together with the compact regularity of the inductive limit representation $\cD'_{L^q} = \bigcup H^{-2m,q}$ this would imply that for
  every $m$ there is $m'$ such that the unit ball of $H^{-2m,q}$ is contained and relatively compact in $H^{-2m', q}$. In case $m'\le m$, the continuous inclusion $H^{-2m',q} \subseteq H^{-2m, q}$ would give that the unit ball of $H^{-2m,q}$ is compact, i.e., that $H^{-2m,q}$ is finite dimensional; in case $m' \ge m$, by tranposition the inclusion $H^{2m,q} \to H^{2m', q}$ and a fortiori the inclusion $H^{2m,q} \to L^q$ would be compact, which cannot be the case \cite[Example 6.11, p.~173]{Adams}.  
%some $m$ the unit ball of $H^{-2m,q}$ is compact which is impossible for an infinite dimensional Banach space.

  (5) $\cD_{L^p, c}$ is a Schwartz space \cite[Def.~5, p.~117]{G3} because it is quasinormable and semi-Montel.

  % (6) The space $\cD_{L^p, c}$ is not nuclear since there are continuous bilinear forms
  % \[ B \colon \cD_{L^p, c, x} \times \cD_{L^p, c, y} \to \bC \]
  % which cannot be represented by a kernel distribution $K \in \cD'_{L^p, c, x, y}$. \footnote{Bsp.?}

  (6) Using the sequence-space representation $\cD_{L^p,c}\cong \ell^p_c\mathbin{\widehat{\otimes}}s$, we can conclude by~\cite[Ch.~II \S 3 n°2, Prop.~13, p.~76]{G} that $\cD_{L^p,c}$ is nuclear if and only if $\ell^p_c$ is nuclear. We first consider the case $p>1$. The nuclearity of $\ell^p_c$ would imply
  \[
    \ell^1\{\ell^p_c\} = \ell^1\mathbin{\widehat{\otimes}_\pi} \ell^p_c = \ell^1\mathbin{\widehat{\otimes}_\varepsilon} \ell^p_c = \ell^1\langle \ell^p_c\rangle
  \]
  where $\ell^1\{\ell^p_c\}$ and $\ell^1\langle \ell^p_c\rangle$ is the space of absolutely summable and of unconditionally summable sequences in $\ell^p_c$, respectively, see~\cite[pp.~341,~359]{Jarchow}.
  We now proceed by giving an example of an element of the space at the very right which is not contained in the space at the very left. Fix $\varepsilon>0$ small enough. Using Hölder's inequality, we observe that
  \[
    \Big\|\Big(\sum_{k=1}^{\infty} \frac{\delta_{jk}}{k^{(1+\varepsilon)/p}}f_k\Big)_j\Big\|_{1} = \Big|\sum_{k=1}^{\infty}\frac{1}{k^{(1+\varepsilon)/p}} f_k\Big| \leq C \|f\|_q
  \]
  for every $f=(f_k)_{k=1}^{\infty}\in \ell^q$ which together with the characterisation of unconditional convergence in~\cite[Theorem~1.15]{Weill} and the condition that compact subsets of $\ell^q$ are small at infinity implies that the sequence is unconditionally convergent.
  % i.e.,
  % \[
  %   \ell^q\to\ell^1, \qquad f\mapsto \Big(\sum_{k=1}^{\infty} \frac{\delta_{jk}}{k^{(1+\varepsilon)/p}}f_k\Big)_{j=1}^{\infty}
  % \]
  % is a continuous linear mapping.
  On the other hand taking $(k^{-\alpha})_{k=1}^{\infty}\in c_0$ with $\alpha= 1-\frac{1+\varepsilon}{p}$ yields
  \[
    \sum_{j=1}^{\infty} \Big\|\delta_{jk} k^{-(1+\varepsilon)/p} g_k)_{k}\|_p = \sum_{j=1}^{\infty} j^{-(1+\varepsilon)/p-\alpha} = \sum_{j=1}^{\infty} \frac{1}{j} = \infty
  \]
  from which we may conclude that $(\delta_{jk}k^{-(1+\varepsilon)/p})_{j,k}$ is not an absolutely summable sequence in $\ell^p_c$.

  For the case $p=1$ we use the Grothendieck-Pietsch criterion, see~\cite[p.~497]{Jarchow}, and observe that $\Lambda(c_{0,+})=\ell^1_c$. Choosing $\alpha=(1/k)_{k=1}^{\infty}$ provides the necessary sequence with $(\alpha_k/\beta_k)_{k=1}^{\infty}\not\in\ell^{1}$ for every $\beta\in c_{0,+}$ with $\beta\geq \alpha$.
\end{proof}

\begin{remark}
  The above proof actually shows that $\ell^p_c$, for $1\leq p\leq \infty$ is not nuclear. From this we may also conclude that $\cD'_{L^p,c}$, $1\leq p\leq\infty$, is not nuclear.
\end{remark}

Analogous to the table with properties of the spaces $\cD^F$, $\cD^{\prime F}$ (defined in \cite[p.~172,173]{H}) in \cite[p.~19]{BD}, we list properties of $\cD_{L^p}$ and $\cD_{L^p,c}$ in the following table ($1 \le p \le \infty$):

\begin{center}
  \begin{tabular}{@{}lll@{}} \toprule
    property & $\cD_{L^p}$ & $\cD_{L^p, c}$ \\
    \midrule
    complete & $+$ & $+$ \\
    quasinormable & $+$ & $+$ \\
    metrizable & $+$ & $-$ \\
    bornological & $+$ & $-$ \\
    barrelled & $+$ & $-$ \\
    (semi-)reflexive & $+$ ($1 < p < \infty$) & $+$ \\
    semi-Montel & $-$ & $+$ \\
    Schwartz & $-$ & $+$ \\
    nuclear & $-$ & $-$ \\ \bottomrule
  \end{tabular}
\end{center}

\section{Sequence space representations of the spaces $\cD_{L^p, c}$, $\cD'_{L^q, c}$ and the compact regularity of $\cD'_{L^q}$}\label{sec4}

P.~and S.~Dierolf conjectured in~\cite[p.~74]{DD} the isomorphy
\[
  \cD_{L^\infty, c} = \cB_c \cong \ell^\infty_c \mathbin{\widehat\otimes} s,
\]
where $\ell^\infty_c = ( \ell^\infty, \kappa ( \ell^\infty, \ell^1 ) ) = ( \ell^\infty, \tau ( \ell^\infty, \ell^1))$. This conjecture is proven in \cite[1.~Theorem, p.~293]{BM}. More generally, we obtain:

\begin{proposition}\label{prop6} Let $1 \le p, q \le \infty$.
  \begin{enumerate}[label=(\alph*)]
  \item $\cD_{L^p,c} \cong \ell^p_c \mathbin{\widehat\otimes} s$;
  \item $\cD'_{L^q, c} \cong \ell^q_c \mathbin{\widehat\otimes} s'$.
  \end{enumerate}
\end{proposition}

\begin{proof}
  (1) By means of M.~Valdivia's isomorphy (see \Cref{rem1} \ref{rem1a}) $\cD'_{L^q} \cong \ell^q \mathbin{\widehat\otimes} s'$ it follows, for $q < \infty$, by \cite[4.1 Theorem, p.~52 and 2.2 Prop., p.~46]{DF}:
  \[
    \cD_{L^p,c} \cong \ell^p_c \mathbin{\widehat\otimes} s,\quad \sfrac{1}{p} + \sfrac{1}{q} = 1.
  \]
  If $p=1$, $\cD_{L^1, c} \cong (\dot\cB')'_c \cong (c_0 \mathbin{\widehat\otimes} s')'_c \cong \ell^1_c \mathbin{\widehat\otimes} s$ by \cite[4.1 Theorem, p.~52 and 2.2 Prop., p.~46]{DF} and \cite[Prop.~7, p.~13]{Ba}. The case $p=\infty$ is a special case of \cite[1.~Theorem, p.~293]{BM}.

  (2) The second Theorem on duality of H.~Buchwalter \cite[(5), p.~302]{K2} yields for two Fr\'echet spaces $E,F$:
  \[
    (E \mathbin{\widehat\otimes_\e} F)'_c \cong E'_c \mathbin{\widehat\otimes_\pi} F'_c,
  \]
  and hence, for $E = \ell^p$ and $F = s$, $E \mathbin{\widehat\otimes} F = \ell^p \mathbin{\widehat\otimes} s$ which implies
  \[
    \cD'_{L^q, c} \cong (\ell^p \mathbin{\widehat\otimes} s)'_c \cong \ell^q_c \mathbin{\widehat\otimes} s'
  \]
  if $1 < q \le \infty$, $\sfrac{1}{p} + \sfrac{1}{q} = 1$. If $q=1$,
  \[
    \cD'_{L^1,c} \cong (\dot\cB)'_c \cong (c_0 \mathbin{\widehat\otimes} s)'_c \cong \ell^1_c \mathbin{\widehat\otimes} s'.
  \]
  An alternative proof for (2) can also by given using \cite[4.1 Theorem, p.~52 and 2.2 Prop., p.~46]{DF} again.
\end{proof}

\begin{remark}
  The sequence space representations of $\cD_{L^p, c}$ and $\cD'_{L^q, c}$ yield a further proof of the quasinormability of these spaces. Indeed, $\ell^p_c$ and $\ell^q_c$ are quasinormable since their duals are Banach spaces and thus, they fulfill the strict Mackey convergence condition \cite[p.~106]{G3}. The claim follows from the fact that the completed tensor product with $s$ or $s'$ remains quasinormable by \cite[Chapter II, Prop.~13.b, p.~76]{G}.
\end{remark}

The compact regularity of $\cD'_{L^q}$ as a countable inductive limit of Banach spaces is proven in Proposition \ref{prop1}. By means of the sequence space representation
\[
  \cD'_{L^q} \cong \ell^q \mathbin{\widehat\otimes} s'
\]
which has been presented in \Cref{rem1} \ref{rem1a} we can give a second proof:

\begin{proposition}\label{prop7}
  Let $E$ be a Banach space.
  \begin{enumerate}[label=(\alph*)]
  \item The inductive limit representations
    \[
      s' \mathbin{\widehat\otimes} E = \varinjlim_k ( ( c_0)_{-k} \mathbin{\widehat\otimes_\e} E) = \varinjlim_k ( c_0 \mathbin{\widehat\otimes_\e} E)_{-k} = \varinjlim_k (c_0(E))_{-k}
    \]
    are valid.
  \item The inductive limit $\varinjlim_k ( c_0(E))_{-k}$ is compactly regular.
  \end{enumerate}
\end{proposition}

\begin{proof}
  (1) The assertion follows from \cite[Theorem 4.1, p.~55]{Ho}.

  (2) The compact regularity of $\varinjlim_k ( c_0(E))_{-k}$ is a consequence of \cite[Thm.~6.4, p.~112]{W1} (or \cite[Thm.~2.7, p.~252]{W2}) and the validity of condition (Q) (see \cite[Prop.~2.3, p.~62]{V}). The condition (Q) reads as:
  \begin{gather*}
    \forall m \in \bN_0\ \exists k>m\ \forall \e>0\ \forall l > k\ \exists C > 0\ \forall x = (x_j)_j \in (c_0(E))_{-m}: \\
    \sup_j j^{-k} \norm{x_j} \le \e \sup_j j^{-m} \norm{x_j} + C \sup_j j^{-l} \norm{x_j}.
  \end{gather*}
  For the sequences $(x_j)_j \in (c_0(E))_{-m}$ the sequence $(\norm{x_j})_j$ is contained in $s' = \varinjlim_k (c_0)_{-k}$. Thus, (Q) is fulfilled because $s'$ is an (LS)-space.
\end{proof}

\begin{corollary} The spaces $\dot{\cB}'\cong\varinjlim_{k} (c_0(c_0))_{-k}$ and $\cD'_{L^q} \cong \varinjlim_k (c_0(\ell^q))_{-k}$, for $1 \le q \le \infty$, are compactly regular countable inductive limits of Banach spaces.
\end{corollary}

\section{``Function''-seminorms in $L^p_c$ and $\cM^1_c$.}\label{sec5}

Motivated by the description of the topology $\kappa ( \ell^p, \ell^q)$ of the sequence space $\ell^p_c$ by the seminorms
\[
  \ell^p \to \bR,\quad x = (x_j)_{j \in \bN} \mapsto p_g(x) = \left( \sum_{j=1}^\infty g(j) \abso{x(j)}^p \right)^{1/p} = \norm{ g \cdot x }_p,
\]
$g \in c_0$, see I.~Mack's master thesis \cite{M} for the case $p=\infty$, we also investigate the description of the topologies $\kappa(L^p, L^q)$ and $\kappa(\cM^1, \cC_0)$ by means of ``function''-seminorms.

Denoting the space of bounded and continuous functions by $\cB^0$ (see \cite[p.~99]{Sch5}) we recall that $\cB^0_c = (\cB^0, \kappa(\cB^0, \cM^1))$ has the topology of R.~C.~Buck.

\begin{proposition}\label{prop8} Let $1 < p \le \infty$.
  The topology $\kappa(L^p, L^q)$ of $L^p_c$ is generated by the family of seminorms
  \[
    L^p \to \bR_+, f \mapsto p_{g,h}(f) = \norm{ h * (gf)}_p,\quad g \in \cC_0, h \in L^1.
  \]
\end{proposition}

\begin{proof}
  
  (a) $t \subseteq \kappa (L^p, L^q)$ where $t$ denotes the topology generated by the seminorms $p_{g,h}$ above and let $\cU_{g,h} \coleq \{ f \in L^p: p_{g,h}(f) \le 1 \}$ be a neighborhood of $0$ in $t$. We show that
  \begin{enumerate}[label=(\roman*)]
  \item $\cU_{g,h} = (g ( \check h * B_{1,q} ) )^\circ$,
  \item $g ( \check h * B_{1,q})$ is compact in $L^q$.
  \end{enumerate}

  ad (i): If $S \in g ( \check h * B_{1,q} ) \subseteq L^q$ there exists $\psi \in B_{1,q}$ with $S = g ( \check h * \psi)$. For $f \in L^p$ with $p_{g,h}(f) \le 1$ we obtain
  \[ \abso{\langle f, S \rangle } = \abso{ \langle f, g ( \check h * \psi) \rangle } = \abso{ \langle h * (gf), \psi \rangle } \]
  and
  \[ \sup_{S \in g ( \check h * B_{1,q} )} \abso{ \langle f, S \rangle } = \norm{ h * (gf)}_p = p_{g,h}(f), \]
  so (i) follows.

  ad (ii): The 3 criteria of the M.~Fr\'echet-M.~Riesz-A.~Kolmogorov-H.~Weyl Theorem are fulfilled:
  \begin{enumerate}[label=(\alph*)]
  \item The set $g(\check h * B_{1,q})$ is bounded in $L^q$.
  \item The set $g(\check h * B_{1,q})$ is small at infinity: $\forall \e>0$ $\exists R>0$ such that $\abso{g(x)} Y(\abso{x}-R)<\e$ for all $x\in\mathbb{R}^{n}$ and hence, 
    \[
      \norm{ Y ( \abso{.} - R) g ( \check h * B_{1,q} ) }_q \le \e \norm{h}_1.
    \]
  \item The set $g(\check h * B_{1,q})$ is $L^q$-equicontinuous:
    \begin{align*}
      \|\tau_s ( g ( \check h * f ) )& - g ( \check h * f ) \|_q \le \\
      &\le \norm{ ( \tau_s g - g)\tau_s ( \check h * f ) }_q + \norm{ g ( \tau_s ( \check h * f ) - ( \check h * f)]}_q \\
      &\le \norm{ \tau_s g - g}_\infty \norm{h}_1 + \norm{g}_\infty\norm{\tau_s \check h - \check h}_1. 
    \end{align*}
    If $s$ tends to $0$, $\norm{\tau_s g - g}_\infty \to 0$ because $g$ is uniformly continuous, and $\norm{\tau_s \check h - \check h}_1 \to 0$ because the $L^1$-modulus of continuity is continuous.
  \end{enumerate}

  (b) $t \supseteq \kappa(L^p, L^q)$:

  For a compact set $C \subseteq L^q$ we have to show that there exist $g \in \cC_0$, $h \in L^1$ such that
  \[ p_{g,h}(f) = \norm{ g ( h * f)}_p \ge \sup_{s \in C} \abso{ \langle f, S \rangle } \]
  for $f \in L^p$. By applying the factorization theorem \cite[(17.1). p.~114]{DW} twice, there exist $g \in \cC_0$, $h \in L^1$ such that $C \subseteq g ( \check h * B_{1,q})$. More precisely, take $L^1$ as the Banach algebra with respect to convolution and $\cC_0$ as the Banach algebra with respect to pointwise multiplication, $L^p$ as the Banach module. As approximate units we use $\mathrm{e}^{-k^2|x|^2}$, $k>0$, in the case of $\cC_0$ and $(4\pi t)^{-n/2}\mathrm{e}^{-|x|^2/4t}$, $t>0$, in the convolution algebra~$L^1$. Then,
  \[
    \sup_{S \in C} \abso{ \langle f, S \rangle } \le \sup_{\psi \in B_{1,q} } \abso{ \langle f, g (\check h * \psi) \rangle } = \sup_{\psi \in B_{1,q}} \abso{ \langle h * (gf), \psi \rangle } = \norm{ h * (gf) }_p,
  \]
  which finishes the proof.
\end{proof}

\begin{proposition}\label{cor8}Let $1 \le q < \infty$ and $C \subseteq L^q$. Then,
  \[
    C\textrm{ compact}\Longleftrightarrow \exists g \in \cC_0\ \exists h \in L^1:\ C \subseteq g ( h * B_{1,q} ).
  \]
\end{proposition}

An exactly analogous reasoning as for Proposition \ref{prop8} yields

\begin{proposition}\label{prop9}The topology $\kappa(\cM^1, \cC_0)$ of the space $\cM^1_c$ is generated by the seminorms
  \[
    \cM^1 \to \bR_+,\ \mu \mapsto p_{g,h}(\mu) \coleq \norm{ h * (g\mu)}_1.
  \]
\end{proposition}

\begin{proof}
  For this, note that $\norm{\mu}_1 = \sup \{ \abso{\langle \varphi, \mu \rangle }: \varphi \in \cC_0, \norm{\varphi}_\infty \le 1 \}$ and that the multiplication
  \[
    \cC_0 \times \cM^1 \to \cM^1,\quad (g,\mu) \mapsto g \cdot \mu
  \]
  and the convolution
  \[
    L^1 \times \cM^1 \to L^1, (h,\nu) \mapsto h * \nu
  \]
  are continuous \cite[Thm~6.4.20, p.~150]{Sch3}).
\end{proof}
  
\begin{proposition}\label{cor9}
  Let $C \subseteq \cM^{1}$. Then,
  \[
    C\textrm{ compact}\Longleftrightarrow \exists g \in \cC_0\ \exists h \in L^1:\ C \subseteq g ( h * B_{1,1} ).
  \]
\end{proposition}

\textbf{Acknowledgments. } E.~A.~Nigsch was supported by the Austrian Science Fund (FWF) grant P26859.

\newcommand{\bibarxiv}[1]{arXiv: \href{http://arxiv.org/abs/#1}{\texttt{#1}}}
\newcommand{\bibdoi}[1]{{\sc doi:} \href{http://dx.doi.org/#1}{\texttt{#1}}}

\printbibliography

\end{document}